\documentclass[12pt, reqno, a4paper]{amsart}
\usepackage[margin=1.2in]{geometry}
\numberwithin{equation}{section}
\usepackage{amssymb,amsfonts,amsthm}
\usepackage[utf8]{inputenc}
\usepackage{listings}
\usepackage{bm}
\usepackage[hyperpageref]{backref}
\usepackage{esint}
\usepackage{color}
\usepackage[bookmarks]{hyperref}
\usepackage{siunitx}
\usepackage{bigints}
\usepackage{hyperref}
\usepackage{mathtools}

\allowdisplaybreaks[4]

\newtheoremstyle{myremark}{10pt}{10pt}{}{}{\bfseries}{.}{.5em}{}

\newtheorem{theorem}{Theorem}[section]
\newtheorem{corollary}[theorem]{Corollary}

\newtheorem{lemma}[theorem]{Lemma}
\newtheorem{proposition}[theorem]{Proposition}

\theoremstyle{definition}

\usepackage{hyperref}

\allowdisplaybreaks[4]

\begin{document}

\title{The Trudinger type inequality in fractional boundary Hardy inequality}

\author{ADIMURTHI, PROSENJIT ROY, AND VIVEK SAHU}

\address{ Department of Mathematics and Statistics,
Indian Institute of Technology Kanpur, Kanpur - 208016, Uttar Pradesh, India}

\email{Adimurthi: adiadimurthi@gmail.com, adimurthi@iitk.ac.in}
\email{Prosenjit Roy: prosenjit@iitk.ac.in}
\email{Vivek Sahu: viveksahu20@iitk.ac.in, viiveksahu@gmail.com}



\date{}

\dedicatory{}

\begin{abstract}
We establish Trudinger-type inequality in the context of fractional boundary Hardy-type inequality for the case ~$sp=d$, where ~$p>1, ~ s \in (0,1)$ on a bounded Lipschitz domain ~$\Omega \subset \mathbb{R}^d$. In particular, we establish fractional version of Trudinger-type inequality with an extra  singular function, namely $d$-th power of the distance  function from $\partial \Omega$ in the denominator of the integrand. The case $d=1$, as it falls in the category $sp=1$, becomes more delicate where  an extra logarithmic correction is required together with subtraction of an average term.  
\end{abstract}

\maketitle


\section{Introduction}

Our aim is to establish a Trudinger-type inequality in fractional boundary Hardy-type inequality for the case ~$sp=d$, with ~$s \in (0,1)$ and ~$p>1$. The classical Hardy-Sobolev inequality is given by: let ~$ p \geq 1, ~\Omega \subset \mathbb{R}^d, ~ d \geq 2$ be a bounded Lipschitz domain with ~$0 \in \Omega$, the following inequality:
\begin{equation*}
    \int_{\Omega} \frac{|u(x)|^{p}}{|x|^{p}} dx \leq \left( \frac{p}{d-p} \right)^{p} \int_{\Omega} |\nabla u(x)|^{p} dx,
\end{equation*}
holds for all ~$u \in C^{\infty}_{c}(\Omega)$ if ~$1 \leq p<d$.  The Hardy-Sobolev inequality with singularity on boundary of ~$\Omega$ for the local case (cf. ~\cite{lewis}) states that there exists a constant ~$C=C(d,p, \Omega)>0$ such that
\begin{equation}\label{boundary hardy}
      \int_{\Omega} \frac{|u(x)|^{p}}{\delta^{p}_{\Omega}(x)}  dx \leq C \int_{\Omega} |\nabla u(x) |^{p}dx, \hspace{.3cm}  \forall \ u \in C^{\infty}_{c}(\Omega),
\end{equation}
where ~$p>1$ and ~$\delta_{ \Omega}$ is the distance function from the boundary of ~$\Omega$ defined by
~$$ \delta_{\Omega}(x):= \underset{y \in \partial \Omega}{\min} |x-y|.$$
B. Dyda in ~\cite{dyda2004} extended ~\eqref{boundary hardy} to fractional Sobolev space and established the following result `among other results': let ~$\Omega \subset \mathbb{R}^d, ~ d \geq 2$ be a bounded Lipschitz domain, ~$p>1, ~ s \in (0,1) $ and ~$sp= d$, then there exists a constant $C=C(d,p, \Omega)>0$ such that
\begin{equation}\label{dyda}
    \int_{\Omega} \frac{|u(x)|^{p}}{\delta^{d}_{\Omega}(x)} dx \leq C \int_{\Omega} \int_{\Omega} \frac{|u(x)-u(y)|^{p}}{|x-y|^{2d}} dxdy =: [u]^{p}_{W^{s,p}(\Omega)}, \hspace{.3cm} \forall \ u \in C^{1}_{c}(\Omega).
\end{equation}
B. Dyda in ~\cite{dyda2004} also gives a counter-example that ~$\eqref{dyda}$ fails when ~$d=1$.  

\smallskip

We define $W^{s,p}_{0}(\Omega)$ as the completion of $C^{\infty}_{c}(\Omega)$ with respect to the norm  $ \|.\|_{W^{s,p}(\Omega)}$ (see ~\eqref{norm defn} for the definition of norm). We need the following  definitions to state our main result. Given any $n_{0} \in \mathbb{N}$, we define the following convex function:
\begin{equation}\label{notation}
   \Phi_{n_{0}}(t) := \exp(t) - \sum_{n=0}^{n_{0}-1} \frac{t^{n}}{n!}, \hspace{5mm} \forall \ t \in \mathbb{R}.
\end{equation}
Let ~$(u)_{\Omega}$ denotes the average of ~$u$ over ~$\Omega$ which is given by
\begin{equation*}
    (u)_{\Omega} := \frac{1}{|\Omega|} \int_{\Omega} u(y)dy.
\end{equation*}

The next theorem presents the main result of this article which establishes fractional version of weighted Trudinger-type inequality due to the presence of singular weight function ~$\frac{1}{\delta^{d}_{\Omega}(x)}$ in the integrand.

 \begin{theorem}\emph{[Main result]}\label{main result}
Let ~$\Omega$ be a bounded Lipschitz domain in $\mathbb{R}^{d}, ~ d \geq 1$ and $ \sup_{x \in \Omega} \delta_{\Omega}(x) = R>0$. Suppose ~$sp=d$ with $p>1$ and $s \in (0,1)$. Let $n_{0} \in \mathbb{N}$ be such that $n_{0}  \geq p -1 > n_{0}-1 $. Then there exists $\alpha_{0}>0$ such that
     \begin{multline}\label{main inequality}
       \sup \left\{  \int_{\Omega} \Phi_{n_{0}} \left( \alpha X_{d}(u)^{\frac{d}{d-s}} \right) \frac{dx}{\delta^{d}_{\Omega}(x)} \ \bigg| \ u \in W^{s,p}_{0}(\Omega), \ [u]_{W^{s,p}(\Omega)} \leq 1  \right\} < \infty, \\ \forall \ \alpha \in [0, \alpha_{0}),    
     \end{multline}
where 
\begin{equation}\label{condition on Xd}
    X_{d}(u) = \begin{dcases}
        \frac{|u(x)-(u)_{\Omega}|}{\ln \left( \frac{2R}{\delta_{\Omega}(x)} \right) } , & d=1 \\ 
        |u(x)|, & d \geq 2.
    \end{dcases}
\end{equation}    
Moreover,
     \begin{multline}\label{main ineq on alpha star}
       \sup \left\{  \int_{\Omega} \Phi_{n_{0}} \left( \alpha X_{d}(u)^{\frac{d}{d-s}} \right) \frac{dx}{\delta^{d}_{\Omega}(x)} \ \bigg| \ u \in W^{s,p}_{0}(\Omega), \ [u]_{W^{s,p}(\Omega)} \leq 1  \right\} = \infty, \\ \forall \ \alpha \in (\alpha^{*}_{d}, \infty),
     \end{multline}     
     where
         \begin{equation}\label{defn alpha star}
         \alpha^{*}_{d} := 
             d  \left(  \frac{2 (d \omega_{d})^{2} \Gamma (p+1) }{d!}  \sum_{n=0}^{\infty} \frac{(d+n-1)!}{n!} \frac{1}{(d+2n)^{p}} \right)^{\frac{s}{d-s}}.
         \end{equation}  
         Here, $\omega_{d}$ denotes the volume of $d$-dimensional unit ball.
 \end{theorem}

 The following corollary establishes that for any ~$u \in W^{s,p}_{0}(\Omega)$ and ~$\alpha > 0$, the function ~$ \Phi_{n_{0}} \left( \alpha X_{d}(u)^{\frac{d}{d-s}} \right) \frac{1}{\delta^{d}_{\Omega}(x)}$ is integrable using the estimate ~\eqref{main inequality} in Theorem ~\ref{main result}. A similar result can be found in ~\cite[Proposition 3.2]{ruf2019}, where E. Parini and B. Ruf considered the function without the weight term involving the distance function ~$\delta_{\Omega}$ and proved the integrability condition for the case ~$d \geq 2$ and ~$sp = d$. By applying the existence of ~$ \alpha_0$ in ~\eqref{main inequality} of Theorem ~\ref{main result}, we can derive the following corollary (see Subsection ~\ref{subsec1} for the proof):
\begin{corollary}\label{corollary 1}
     Let ~$\Omega$ be a bounded Lipschitz domain in ~$\mathbb{R}^{d}, ~ d \geq 1$ and $sp=d$. Let $n_{0} \in \mathbb{N}$ be such that ~$n_{0}  \geq p -1 > n_{0}-1 $ and ~$u \in W^{s,p}_{0}(\Omega)$. For any ~$\alpha>0$, we have
    \begin{equation}
     \Phi_{n_{0}} \left( \alpha X_{d}(u)^{\frac{d}{d-s}} \right) \frac{1}{\delta^{d}_{\Omega}(x)} \in L^{1}(\Omega). 
    \end{equation}
\end{corollary}

 \smallskip

We remark that, Theorem ~\ref{main result} can be extended a  little further  by  adding an extra  

$$   \bigintssss_{\Omega} \sum_{n=0}^{n_{0}-1}  \frac{\alpha^{n} X_{d}(u)^{\frac{nd}{d-s}}}{n!} dx  $$ in \eqref{main inequality}.  More precisely, one can show that 

  \begin{equation}\label{extension main result}
  \begin{split}
       \sup \Bigg\{  \bigintsss_{\Omega} \Bigg( \sum_{n=0}^{n_{0}-1}  \frac{\alpha^{n} X_{d}(u)^{\frac{nd}{d-s}}}{n!} + \frac{\Phi_{n_{0}} \left( \alpha X_{d}(u)^{\frac{d}{d-s}} \right)}{\delta^{d}_{\Omega}(x)} \Bigg) dx  \ \bigg|  \  u \in  W^{s,p}_{0}(\Omega), &  \ [u]_{W^{s,p}(\Omega)} \leq 1  \Bigg\}  \\ &  <  \infty, \ \forall \ \alpha \in [0, \alpha_{0}),
  \end{split}
     \end{equation}
where ~$n_{0} \geq p-1> n_{0}-1$ and $\alpha_{0}$ is as defined in Theorem ~\ref{main result}. Also,  the above supremum is infinity for $\alpha > \alpha_d^*.$ Proof of this remark (Proof of remark) is done at the very end of this article. Since we are working on bounded domains, using the trivial estimate $\delta_\Omega(x) < R$, for all $x \in \Omega$ and for some $R>0$,  the main result \cite[Theorem 1.1]{ruf2019} by E. Parini and B. Ruf  follows directly from  \eqref{extension main result} and the line after that.

\smallskip

 The choice of $\alpha^{*}_{d}$ in the above theorem for $d \geq 2$ follows from the fractional version of the Moser-Trudinger inequality by E. Parini and B. Ruf ~\cite[Theorem $1.1$]{ruf2019} and the inequality
\begin{equation*}
  \int_{\Omega} \Phi_{n_{0}} \left( \alpha X_{d}(u)^{\frac{d}{d-s}} \right) dx \leq C \int_{\Omega} \Phi_{n_{0}} \left( \alpha X_{d}(u)^{\frac{d}{d-s}} \right)  \frac{dx}{\delta^{d}_{\Omega}(x)} , 
\end{equation*}
where $C>0$ is such that $\delta^{d}_{\Omega}(x) < C$ for all $x \in \Omega$. By applying \cite[Theorem $1.1$]{ruf2019}, we conclude that \eqref{main inequality} also fails when $d \geq 2$ for any ~$\alpha> \alpha^{*}_{d}$. For the case $d=1$, $\alpha^{*}_{1}$ is obtained from the work of ~\cite[Theorem $1.1$]{lula2017} where S. Iula considered the fractional version of Moser-Trudinger inequality in dimension one.

\smallskip

The above theorem can be interpreted as Trudinger-type theorem in the context of non local boundary Hardy Inequality. Notice that the form of $X_1$ is very different from $X_d$ for $d\geq 2$. To explain this, we first delve into the question why \eqref{dyda} fails for $d=1$.  Take $\Omega = (0,1)$. It is well known when ~$sp=1$, ~$W^{s,p}(\Omega) = W^{s,p}_{0}(\Omega)$ (see ~\cite[Theorem 6.78]{leonibook}), which implies that the constant functions are in $ W^{s,p}_{0}(\Omega)$. Clearly \eqref{dyda} fails for constant functions as the left hand side is singular and the right hand side is zero. To stay away from the space of constant functions the  average of the function is subtracted in $X_1$. In fact for any smooth non constant functions in $W^{s,p}(0,1)$ that takes constant values near $x=0$ or $x=1$, the left hand side of \eqref{dyda} is singular and this singularity is killed by multiplying the integrand on the left hand side by $\frac{1}{\ln^{p}}$ function. 

\smallskip 

The main idea of the proof of Theorem \ref{main result} is to show (see Proposition \ref{prop}) that there exists a positive constant $C$ that ``does not" depend on ~$\tau$  such that for all $\tau \geq p$, 
    \begin{equation*}
      \left(  \int_{\Omega} (X_{d}(u))^{\tau} \frac{dx}{\delta^{d}_{\Omega}(x)} \right)^{\frac{1}{\tau}} \leq C \tau^{\frac{d-s}{d}} [u]_{W^{s,p}(\Omega)}, \hspace{.3cm} \ \forall \ u \in W^{s,p}_{0}(\Omega).
    \end{equation*}
 Then the result follows by appropriately summing up these inequalities.

\smallskip

\noindent \textbf{Applications:} K. Perera and M. Squassina ~\cite{marco2018} studied the eigenvalue problem associated with fractional Moser-Trudinger nonlinearities, proving the existence of nontrivial solutions and exploring certain properties of distinct pairs of nontrivial solutions to the problem. In a related work, N. V. Thin ~\cite{thin2020} investigated the fractional singular Trudinger inequality and illustrate the existence of nontrivial solutions for the fractional $p$-Laplacian involving singular Moser-Trudinger inequality. Applying Theorem \ref{main result}, it would be interesting to investigate similar problems within the framework of Trudinger-type inequalities in the context of the fractional boundary Hardy inequality.

\smallskip

The available literature is huge on Hardy inequality and Trudinger-Moser type inequalities  and we refer to some of the related work in ~\cite{adimurthi2024submanifold, dyda2004, lula2017, LAM2020108673, ruf2019, marco2018,  sahu2024weighted, thin2020}.  Our work in this article can be considered as weighted version of fractional Trudinger-type inequality. For weighted Moser-Trudinger inequality in local version, we refer to ~\cite{adi2007, calanchi2015, gyula2021, gyula2015, gyula2016, roy2019}.

\smallskip

The article is organized in the following way: In Section ~\ref{preliminaries}, we present preliminary lemmas and notations that will be utilized to prove Proposition \ref{prop}. In Section \ref{fractional hardy}, we prove Proposition \ref{prop}  which follows as a consequence  of  Lemma \ref{flat case sp=1} and Lemma \ref{sp=d flat case} after the usual  patching up of the boundary technique. The proof of both the lemmas above can also be found in ~\cite{AdiPurbPro2023, adimurthi2023fractional}, we choose to present them here for the completeness of this article. Section ~\ref{proof of main result} contains the proof of the main theorem, Theorem ~\ref{main result}, which  just requires Proposition \ref{prop} as main ingredient. 


\section{Notations and Preliminaries}\label{preliminaries} 
In this section, we introduce the notations and preliminary lemmas that will be used in proving the Proposition ~\ref{prop} and Theorem ~\ref{main result}. Throughout this article, we shall use the following notations: 
\begin{itemize}
    \item $s$ will always be understood to be in ~$(0,1)$.
    \item we denote ~$|\Omega|$ the Lebesgue measure of ~$\Omega \subset \mathbb{R}^{d} $.
    \item $C>0$ will denote a generic constant that may change from line to line. 
\end{itemize}
\smallskip
 For any ~$p \in [1, \infty)$ and $s \in (0,1)$, define the fractional Sobolev space
\begin{equation*}
W^{s,p}(\Omega) := \left\{    u \in L^{p}(\Omega) : \int_{\Omega} \int_{\Omega}  \frac{|u(x)-u(y)|^{p}}{|x-y|^{d+sp}}  dxdy < \infty \right\},
\end{equation*}
endowed with the norm
\begin{equation}\label{norm defn}
    \|{u}\|_{W^{s,p}(\Omega)} := \left(  [u]^{p}_{W^{s,p}(\Omega)}  + \|{u}\|^{p}_{L^{p}(\Omega)}  \right)^{\frac{1}{p}}.
\end{equation}

\smallskip

{\bf Fractional Poincar\'e\ Inequality} ~\cite[Theorem 3.9]{edmunds2022}: Let ~$p \geq 1$ and ~$\Omega \subset \mathbb{R}^d$ be a bounded open set. Then there exists a constant ~$C = C(d,p,s, \Omega)>0$ such that
\begin{equation}\label{poincare}
    \int_{\Omega} |u(x)-(u)_{\Omega}|^{p}  dx \leq C [u]^{p}_{W^{s,p}(\Omega)} .
\end{equation} 

\smallskip

One has, for any ~$a_{1},  \dots , a_{m} \in \mathbb{R}$ and ~$\gamma \geq1$,
\begin{equation}\label{sumineq}
    \sum_{\ell=1}^{m} |a_{\ell}|^{\gamma} \leq  \left( \sum_{\ell=1}^{m} |a_{\ell}| \right)^{\gamma} .
\end{equation}

\smallskip

The next lemma establishes a fractional Sobolev inequality with parameter ~$\lambda>0$.  A similar type of lemma is available in ~\cite{AdiPurbPro2023, adimurthi2023fractional, squassina2018}.  This lemma is helpful in proving Lemma ~\ref{flat case sp=1} and Lemma ~\ref{sp=d flat case}.

\begin{lemma}\label{sobolev}
   Let ~$\Omega$ be a bounded Lipschitz domain in ~$\mathbb{R}^{d}, ~ d \geq 1$. Let ~$p > 1$ and ~$s \in (0,1)$ such that ~$sp = d$. Define ~$\Omega_{\lambda}: = \left\{ \lambda x : ~ x \in \Omega \right\}$ for ~$\lambda>0$, then there exists a positive constant ~$C=C(d,p,s,\Omega) $ such that, for any ~$\tau \geq p$, we have
    \begin{equation}
         \left( \fint_{\Omega_{\lambda}} |u(x)-(u)_{\Omega_{\lambda}}|^{\tau } dx \right)^{\frac{1}{\tau}}  \leq C  \tau^{\frac{d-s}{d}} [u]_{W^{s,p}(\Omega_{\lambda})}, \hspace{3mm} \forall \ u \in W^{s,p}(\Omega_{\lambda}).
    \end{equation}
    \end{lemma} 
\begin{proof}
 Let ~$\Omega \subset \mathbb{R}^d$ be a bounded Lipschitz domain and $u \in W^{s,p}(\Omega)$. Then from the extension theorem in ~\cite[Theorem $5.4$]{di2012hitchhikers} there exists $\tilde{u} \in W^{s,p}(\mathbb{R}^{d})$ such that $\tilde{u}|_{\Omega} =  u$ and 
 \begin{equation*}
     \| \tilde{u} \|_{W^{s,p}(\mathbb{R}^{d})} \leq C \|u\|_{W^{s,p}(\Omega)},
 \end{equation*} 
where $C=C(d,p,s, \Omega)>0$. By \cite[Theorem $7.13$]{leonibook} (see also \cite[Theorem 9.1]{Peetr1966}), for any ~$\tau \geq p$ if ~$sp=d$, we have
\begin{equation*}
      \| \tilde{u}\|_{L^{\tau}(\mathbb{R}^{d})} \leq C \tau^{\frac{d-s}{d}+ \frac{1}{\tau}} \|\tilde{u}\|_{W^{s,p}(\mathbb{R}^{d})},
  \end{equation*}
   where ~$C=C(d,p,s, \Omega)$ is a positive constant. Combining the above two inequalities with $\| u \|_{L^{\tau}(\Omega)} \leq \| \tilde{u}\|_{L^{\tau}(\mathbb{R}^{d})}$, we obtain
   \begin{equation*}
        \| u\|_{L^{\tau}(\Omega)} \leq C \tau^{\frac{d-s}{d}+ \frac{1}{\tau}} \|u\|_{W^{s,p}(\Omega)}.
   \end{equation*}
    Since, $\tau^{\frac{1}{\tau}}$ is bounded for any $\tau \geq p$. Therefore, we have
    \begin{equation}\label{Sobineqq}
        \|u\|_{L^{\tau}(\Omega)} \leq C \tau^{\frac{d-s}{d}} \|u\|_{W^{s,p}(\Omega)}.
   \end{equation}
   Applying ~\eqref{Sobineqq} with ~$u-(u)_{\Omega}$ and using ~\eqref{poincare}, we have
    \begin{equation*}
        \|u-(u)_{\Omega}\|_{L^{\tau}(\Omega)} \leq C \tau^{\frac{d-s}{d}} [u]_{W^{s,p}(\Omega)} .
    \end{equation*}
    Let us apply the above inequality to ~$u(\lambda x)$ instead of ~$u(x)$. This gives
    \begin{equation*}
        \left(  \fint_{\Omega}  \Big|u(\lambda x)-\fint_{\Omega} u(\lambda x) dx \Big|^{\tau } \ dx \right)^{\frac{1}{\tau}}  \leq C  \tau^{\frac{d-s}{d}} \left( \int_{\Omega} \int_{\Omega} \frac{|u(\lambda x) - u(\lambda y)|^{p}}{|x-y|^{2d}} dxdy  \right)^{\frac{1}{p}}  . 
    \end{equation*}
    Using the fact
    \begin{equation*}
        \fint_{\Omega} u(\lambda x) dx = \fint_{\Omega_{\lambda}} u(x) dx,
    \end{equation*}
    we have
    \begin{equation*}
       \left(  \fint_{\Omega} |u(\lambda x)-(u)_{\Omega_{\lambda}}|^{\tau } \ dx \right)^{\frac{1}{\tau}}  \leq C  \tau^{\frac{d-s}{d}} \left( \int_{\Omega} \int_{\Omega} \frac{|u(\lambda x) - u(\lambda y)|^{p}}{|x-y|^{2d}} dxdy  \right)^{\frac{1}{p}}  . 
    \end{equation*}
    By changing the variable ~$X=\lambda x$ and ~$Y= \lambda y$, we obtain 
    \begin{equation*}
         \left( \fint_{\Omega_{\lambda}} |u(x)-(u)_{\Omega_{\lambda}}|^{\tau } dx \right)^{\frac{1}{\tau}}  \leq C  \tau^{\frac{d-s}{d}}   [u]_{W^{s,p}(\Omega_{\lambda})}   . 
    \end{equation*}
    This finishes the proof of the lemma.
\end{proof}

\smallskip

The next lemma establishes fractional Sobolev inequalities for the case $sp=d$ and $\tau \geq 1$. This inequalities plays a crucial role in establishing our main result.

\begin{lemma}\label{lemma on tau < p}
  Let ~$\Omega$ be a bounded Lipschitz domain in $\mathbb{R}^{d}$. Let $p>1$ and $s \in (0,1)$ such that $sp=d$. Then for any $\tau \geq 1$, there exists a constant $C=C(d,p,s, \Omega)>0$ such that for $d \geq 2$,
   \begin{equation}\label{sob ineq d geq 2}
        \|u\|_{L^{\tau}(\Omega)} \leq C \tau^{\frac{d-s}{d}} [u]_{W^{s,p}(\Omega)}, \hspace{3mm} \forall \ u \in W^{s,p}_{0}(\Omega).
    \end{equation}
  In the case $d=1$, 
     \begin{equation}\label{sob ineq d=1}
       \| u-(u)_{\Omega}\|_{L^{\tau}(\Omega)} \leq C \tau^{1-s} [u]_{W^{s,p}(\Omega)},  \hspace{3mm} \forall \ u \in W^{s,p}_{0}(\Omega).
    \end{equation}
\end{lemma}
\begin{proof}
Let ~$\Omega \subset \mathbb{R}^d$ be a bounded Lipschitz domain and $u \in W^{s,p}_{0}(\Omega)$.  From fractional boundary Hardy inequality for the case ~$sp = d$ and ~$d \geq 2$ (refer to ~\eqref{dyda}), we have the inequality
\begin{equation*}
  \int_{\Omega} |u(x)|^{p} dx \leq C \int_{\Omega} \frac{|u(x)|^{p}}{\delta^{d}_{\Omega}(x)} dx \leq C [u]^{p}_{W^{s,p}(\Omega)},
\end{equation*}
 where $C=C(d,p,s, \Omega)>0$. Applying this inequality to ~\eqref{Sobineqq} for $\tau \geq p$ and $d \geq 2$, we conclude that there exists a constant $C=C(d,p,s, \Omega)>0$ such that
   \begin{equation}\label{eqn11111}
        \|u\|_{L^{\tau}(\Omega)} \leq C \tau^{\frac{d-s}{d}} [u]_{W^{s,p}(\Omega)}.
    \end{equation}
   This establishes \eqref{sob ineq d geq 2} for $\tau \geq p$.  Now assume $\tau <p$. By applying H$\ddot{\text{o}}$lder's inequality with exponent $\frac{p}{\tau}>1 $  and $\frac{p}{p- \tau}>1$, we get
    \begin{equation*}
     \left(   \int_{\Omega} |u(x)|^{\tau} dx \right)^{\frac{1}{\tau}} \leq \left( \int_{\Omega} |u(x)|^{p} dx  \right)^{\frac{1}{p}} |\Omega|^{\frac{1}{\tau}- \frac{1}{p}}.
    \end{equation*}
    By applying the inequality \eqref{eqn11111}  for $\tau=p$ and using $|\Omega|^{\frac{1}{\tau} - \frac{1}{p}} \leq \max \{ 1, |\Omega|^{1-\frac{1}{p}} \}$, we obtain
    \begin{equation*}
        \|u\|_{L^{\tau}(\Omega)} \leq C [u]_{W^{s,p}(\Omega)},
    \end{equation*}
   where $C= C(d,p,s, \Omega)>0$. Moreover, $\tau^{\frac{d-s}{d}} \geq 1 $,  we obtain
   \begin{equation*}
        \| u \|_{L^{\tau}(\Omega)} \leq C \tau^{\frac{d-s}{d}} [u]_{W^{s,p}(\Omega)}.
 \end{equation*}
Furthermore, when $d=1$. From Lemma ~\ref{sobolev} with $\lambda=1$ and $u \in W^{s,p}_{0}(\Omega)$, we have for any $\tau \geq p$,
\begin{equation*}
    \|u-(u)_{\Omega}\|_{L^{\tau}(\Omega)} \leq C \tau^{1-s} [u]_{W^{s,p}(\Omega)}.
\end{equation*}
Similarly using H$\ddot{\text{o}}$lder's inequality for the case $\tau <p$ with exponent  $\frac{p}{\tau}>1 $  and $\frac{p}{p- \tau}>1$ and then applying fractional Poincar\'e inequality (see ~\eqref{poincare}) for the case $d=1$, we arrive at
\begin{equation*}
      \| u-(u)_{\Omega}\|_{L^{\tau}(\Omega)} \leq C [u]_{W^{s,p}(\Omega)},
\end{equation*}
 where $C= C(d,p,s, \Omega)>0$. Moreover, $\tau^{1-s} \geq 1 $. Therefore, for $\tau<p$, we obtain
    \begin{equation*}
       \|u-(u)_{\Omega}\|_{L^{\tau}(\Omega)} \leq C \tau^{1-s} [u]_{W^{s,p}(\Omega)}.
    \end{equation*}
    This proves the lemma.
\end{proof}

\smallskip

 The following lemma establishing a connection to the average of ~$u$ over two disjoint sets. A similar type of lemma is also available in ~\cite{adimurthi2023fractional}.   This lemma is helpful in proving Lemma ~\ref{flat case sp=1} and Lemma ~\ref{sp=d flat case}.
\begin{lemma}\label{avg}
    Let ~$E$ and ~$F$ be disjoint set in ~$\mathbb{R}^d$. Then for any ~$\tau \geq 1$, we have
    \begin{equation}
        |(u)_{E} - (u)_{F}|^{\tau} \leq 2^{\tau} \frac{|E \cup F|}{\min \{ |E|, |F| \} }  \fint_{E \cup F} |u(x)-(u)_{E \cup F}|^{\tau}dx  . 
    \end{equation}
\end{lemma}
\begin{proof} 
Let us consider ~$|(u)_{E}-(u)_{F}|^{\tau}$, we have
\begin{equation*}
\begin{split}
    |(u)_{E}-(u)_{F}|^{\tau}  & \leq 2^{\tau} |(u)_{E} - (u)_{E \cup F}|^{\tau} +  2^{\tau} | (u)_{F} + (u)_{E \cup F}|^{\tau} \\
       & = 2^{\tau} \Big| \fint_{E} \left\{ u(x) - (u)_{E \cup F} \right\}   dx  \Big|^{\tau} + 2^{\tau} \Big| \fint_{F} \left\{ u(x) - (u)_{E \cup F} \right\}  dx \Big|^{\tau}. 
\end{split}
\end{equation*}
    By using H$\ddot{\text{o}}$lder's inequality with ~$ \frac{1}{\tau} +  \frac{1}{\tau'} = 1$, we have
    \begin{equation*}
    \begin{split}
     |(u)_{E}-(u)_{F}|^{\tau}  &\leq  2^{\tau} \fint_{E} |u(x) - (u)_{E \cup F} |^{\tau}   dx + 2^{\tau} \fint_{F} | u(x) - (u)_{E \cup F} |^{\tau}  dx \\
       &\leq  \frac{2^{\tau}}{\text{min} \{ |E|, |F| \} } \int_{E \cup F} |u(x) - (u)_{E \cup F} |^{\tau}   dx \\&
       = 2^{\tau} \frac{|E \cup  F|}{\text{min} \{ |E|, |F| \}} \fint_{E \cup F} |u(x) - (u)_{E \cup F} |^{\tau}   dx .
    \end{split}
    \end{equation*}
    This finishes the proof of the lemma.
\end{proof}
\smallskip

The next lemma establishes a basic inequality for large ~$n \in \mathbb{N}$. This lemma will be helpful in establishing that the positive constant ~$C$ in Proposition ~\ref{prop} does not depend on ~$\tau$.

\begin{lemma}\label{large n ineq}
Let ~$p>1, ~ \tau \geq p$. Then for large ~$n \in \mathbb{N}$, we have
     \begin{equation}
         \frac{1}{(n)^{\tau-1}} - \frac{1}{ \left(n + \frac{1}{2} \right)^{\tau-1}} \geq \left( \frac{p-1}{4} \right) \frac{1}{n^{\tau}}.
    \end{equation}
\end{lemma}
\begin{proof}
 See ~\cite[Lemma 2.6]{adimurthi2023fractional} for the proof.
\end{proof}

\section{Fractional boundary Hardy inequality}\label{fractional hardy}

The next lemma proves the Proposition ~\ref{prop} when ~$\Omega= (0,1)$ for the case $d =1$ and the test functions are supported on ~$\left(0, \frac{1}{2} \right)$. The following lemma is available in ~\cite{AdiPurbPro2023}. For the sake of completeness of this article, we provide a proof here, explicitly highlighting the dependence on the parameter $\tau$ in the resulting constant.

\begin{lemma}\label{flat case sp=1}
     Let   ~$sp=d=1$ and ~$\tau \geq p$, then there exists a constant ~$C= C(d,p) > 0$  such that
    \begin{equation}
     \left(  \int_{0}^{\frac{1}{2}} \frac{|u(x)|^{\tau}}{x \ln^{\tau} \left(\frac{1}{x} \right)}  dx \right)^{\frac{1}{\tau}} \leq  C  \tau^{1-s} 
 \|u\|_{W^{s,p}((0,1))}, \hspace{3mm} \forall \ u \in C^{1}_{c}((0,1)).
    \end{equation}
\end{lemma}
\begin{proof}
For each ~$k \leq -1$, set $A_{k} := \{ x :  \ 2^{k} \leq x < 2^{k+1} \}$. Applying Lemma ~\ref{sobolev} with ~$\Omega = (1,2)$, ~$ \lambda = 2^k$ and $sp=d=1$, we have
\begin{equation*}
    \fint_{A_{k}} |u(x)-(u)_{A_{k}}|^{\tau}  dx \leq C^{\tau}\tau^{(1-s) \tau} [u]^{\tau}_{W^{s,p}(A_{k})}  ,
\end{equation*}
where ~$C= C(d,p)$ is a positive constant. Since, ~$x \geq 2^k$ which implies ~$ \frac{1}{x} \leq \frac{1}{2^{k}}$. Therefore, using this and previous inequality, we have
\begin{equation*}
\begin{split}
    \int_{A_{k}} \frac{|u(x)|^{\tau}}{x}  dx &  \leq \frac{2^{\tau}}{2^{k }} \int_{A_{k}} |u(x)-(u)_{A_{k}}|^{\tau}   dx + \frac{2^{\tau}}{2^{k }} \int_{A_{k}} |(u)_{A_{k}}|^{\tau}dx \\ & \leq C^{\tau} \tau^{(1-s) \tau} [u]^{\tau}_{W^{s,p}(A_{k})} + 2^{\tau} |(u)_{A_{k}}|^{\tau}, 
\end{split}
    \end{equation*}
where ~$C$ is a positive constant depends on ~$d$ and ~$p$. For each ~$x \in A_{k}$, we have ~$\frac{1}{x} > \frac{1}{2^{k+1}}$ which  implies ~$ \ln \left(\frac{1}{x} \right) > (-k-1) \ln 2$. Using this and ~$\frac{1}{(-k-1)^{\tau}} \leq 1$, where $k \leq -2$, we obtain
\begin{equation*}
\begin{split}
    \int_{A_{k}} \frac{|u(x)|^{\tau}}{x \ln^{\tau} \left(\frac{1}{x} \right)}  dx & \leq \frac{C^{\tau} \tau^{(1-s) \tau}}{(-k-1)^{\tau}}  [u]^{\tau}_{W^{s,p}(A_{k})} + \frac{C^{\tau}}{(-k-1)^{\tau}} |(u)_{A_{k}}|^{\tau} 
    \\ & \leq  C^{\tau} \tau^{(1-s) \tau} [u]^{\tau}_{W^{s,p}(A_{k})} +  \frac{C^{\tau}}{(-k-1)^{\tau}} |(u)_{A_{k}}|^{\tau}  .
\end{split}
\end{equation*}
Summing the above inequality from ~$k=m \in \mathbb{Z}^{-}$ to ~$-2$, we get
\begin{equation}\label{eqnn2}
\sum_{k=m}^{-2} \int_{A_{k}} \frac{|u(x)|^{\tau}}{x \ln^{\tau}\left( \frac{1}{x} \right)} dx \leq  C^{\tau}  \tau^{(1-s) \tau} \sum_{k=m}^{-2} [u]^{\tau}_{W^{s,p}(A_{k})} + C^{\tau} \sum_{k=m}^{-2}  \frac{|(u)_{A_{k}}|^{\tau}}{(-k-1)^{\tau}}     .
\end{equation}
Independently, using triangle inequality, we have 
\begin{equation*}
    |(u)_{A_{k}}|^\tau \leq  \left( |(u)_{A_{k+1}}| + |(u)_{A_{k}} - (u)_{A_{k+1}}| \right)^\tau.
\end{equation*}
For ~$\tau >1, ~ c>1$ and ~$a, ~ b \in \mathbb{R}$, we have (see ~\cite[Lemma 2.5]{adimurthi2023fractional})
    \begin{equation}\label{estimate}
        (|a| + |b|)^{\tau} \leq c|a|^{\tau} + (1-c^{\frac{-1}{\tau -1}})^{1-\tau} |b|^{\tau} .
        \end{equation}
For ~$k \in \mathbb{Z}^{-} \backslash \{-1\}$, applying the above inequality with ~$c := \left( \frac{-k-1}{-k-(3/2)} \right)^{\tau-1}>1$ with ~$sp=1$ and using Lemma ~\ref{avg} with ~$E=A_{k}$ and $F=A_{k+1}$ together with Lemma ~\ref{sobolev}, we obtain
\begin{equation*}
    |(u)_{A_{k}}|^{\tau} \leq \left( \frac{-k-1}{-k-(3/2)} \right)^{\tau-1} |(u)_{A_{k+1}}|^{\tau} + C^{\tau} \tau^{(1-s) \tau} (-k-1)^{\tau-1} [u]^{\tau}_{W^{s,p}(A_{k} \cup A_{k+1})}  .
\end{equation*}
Summing the above inequality from ~$k=m \in \mathbb{Z}^{-}$ to ~$-2$, we get
\begin{equation*}
   \sum_{k=m}^{-2}  \frac{|(u)_{A_{k}}|^{\tau}}{(-k-1)^{\tau-1}} \leq \sum_{k=m}^{-2} \frac{ |(u)_{A_{k+1}}|^{\tau}}{(-k-(3/2))^{\tau-1}} + C^{\tau} \tau^{(1-s) \tau} \sum_{k=m}^{-2} [u]^{\tau}_{W^{s,p}(A_{k} \cup A_{k+1})}  .
\end{equation*}
By changing sides, rearranging, and re-indexing, we get
\begin{equation*}
\begin{split}
    \frac{|(u)_{A_{m}}|^{\tau}}{(-m-1)^{\tau-1}} +  \sum_{k=m+1}^{-2} & \left\{ \frac{1}{(-k-1)^{\tau-1}} - \frac{1}{(-k-1/2)^{\tau-1}} \right\} |(u)_{A_{k}}|^{\tau}  \\ & \leq  2^{\tau-1} |(u)_{A_{-1}}|^{\tau} 
    + \ C^{\tau} \tau^{(1-s) \tau} \sum_{k=m}^{-2} [u]^{\tau}_{W^{s,p}(A_{k} \cup A_{k+1})}  .   
\end{split}
\end{equation*}
Now, for large values of ~$-k$,  using the asymptotics (see Lemma ~\ref{large n ineq})
\begin{equation*}
     \frac{1}{(-k-1)^{\tau-1}} - \frac{1}{(-k-1/2)^{\tau-1}} \sim \frac{1}{(-k-1)^{\tau}},
\end{equation*}
and  choose large $-m$ such that ~$ (u)_{A_{m}} = 0 $, we get
\begin{equation}\label{eqn10}
    \sum_{k=m}^{-2} \frac{|(u)_{A_{k}}|^{\tau}}{(-k-1)^{\tau}} \leq  C^{\tau}  |(u)_{A_{-1}}|^{\tau}  +  C^{\tau} \tau^{(1-s) \tau} \sum_{k=m}^{-2} [u]^{\tau}_{W^{s,p}(A_{k} \cup A_{k+1})}  .
\end{equation}
Combining ~\eqref{eqnn2} and ~\eqref{eqn10} yields
\begin{equation*}
    \sum_{k=m}^{-2} \int_{A_{k}} \frac{|u(x)|^{\tau}}{x \ln^{\tau} \left( \frac{1}{x} \right)}  dx \leq  C^{\tau}  |(u)_{A_{-1}}|^{\tau} +  C^{\tau} \tau^{(1-s) \tau} \sum_{k=m}^{-2} [u]^{\tau}_{W^{s,p}(A_{k} \cup A_{k+1})}  .
\end{equation*}
Also, ~$ |(u)_{A_{-1}}|^{\tau} \leq C^{\tau}  \|u\|^{\tau}_{L^{p}((0,1))}$. Hence, combining this in the above inequality and using $\tau^{(1-s) \tau} > 1$, we obtain
\begin{equation*}
     \sum_{k=m}^{-2} \int_{A_{k}} \frac{|u(x)|^{\tau}}{x \ln^{\tau} \left( \frac{1}{x} \right)}  dx \leq  C^{\tau} \tau^{(1-s) \tau} \Bigg( \|u\|^{\tau}_{L^{p}((0,1))} +   \sum_{k=m}^{-2} [u]^{\tau}_{W^{s,p}(A_{k} \cup A_{k+1})} \Bigg).
\end{equation*}
Applying the inequality ~\eqref{sumineq} with $\gamma = \frac{\tau}{p}$, we have
\begin{equation*}
\begin{split}
\sum_{k=m}^{-2}  [u]^{\tau}_{W^{s,p}(A_{k} \cup A_{k+1})} &=    \Bigg( \sum_{\substack{k =m \\ -k \ \text{is even}}}^{-2} [u]^{\tau}_{W^{s,p}(A_{k} \cup A_{k+1})} + \sum_{\substack{k =m \\ -k \ \text{is odd}}}^{-2} [u]^{\tau}_{W^{s,p}(A_{k} \cup A_{k+1})} \Bigg) \\ & \leq 2[u]^{\tau}_{W^{s,p}((0,1))}.
\end{split}
\end{equation*}
Therefore, combining the above two inequalities and using  ~$\tau \geq p$, we obtain
 \begin{equation*}
 \begin{split}
      \left( \int_{0}^{\frac{1}{2}} \frac{|u(x)|^{\tau}}{x \ln^{\tau}\left( \frac{1}{x} \right)}  dx \right)^{\frac{1}{\tau}} &\leq C \tau^{1-s} \left( [u]^{\tau}_{W^{s,p}((0,1))} + \|u\|^{\tau}_{L^{p}((0,1))} \right)^{\frac{1}{\tau}}  \\
       &\leq C \tau^{1-s} \left( [u]^{p}_{W^{s,p}((0,1))} + \|u\|^{p}_{L^{p}((0,1))} \right)^{\frac{1}{p}} .
    \end{split}
 \end{equation*}
 This proves the lemma.
\end{proof}

The next lemma proves Proposition ~\ref{prop} for the case $d \geq 2$ when the domain is ~$\mathbb{R}^{d}_{+}$ and the test functions are supported on cubes ~$\Omega_{n}:= (-n,n)^{d-1} \times (0,1)$, where $n \in \mathbb{N}$. This lemma is available in ~\cite{adimurthi2023fractional}. For the sake of completeness of this article, we provide a proof here, explicitly highlighting the dependence on the parameter $\tau$ in the resulting constant. 

\begin{lemma}\label{sp=d flat case}
     Let ~$sp=d, ~d\geq 2$ and ~$\tau \geq p$, there exist a constant ~$C= C(d,p) > 0$  such that
    \begin{equation}
     \left(  \int_{\Omega_{n}} \frac{|u(x)|^{\tau}}{x^{d}_{d} }  dx \right)^{\frac{1}{\tau}} \leq  C  \tau^{\frac{d-s}{d}}
 [u]_{W^{s,p}(\Omega_{n})}, \hspace{3mm} \forall \ u \in C^{1}_{c}(\Omega_n).
    \end{equation}
\end{lemma}
\begin{proof}
 For each ~$k \leq -1$, set
\begin{equation*}
    A_{k} := \{ (x',x_{d}) : \  x' \in (-n,n)^{d-1}, \ 2^{k} \leq x_{d} < 2^{k+1} \} .
\end{equation*}
Then, we have $ \Omega_{n}= \bigcup_{k = - \infty}^{-1} A_{k}$. Again, we  further divide each ~$A_{k}$ into disjoint cubes each of side length ~$2^k$ (say ~$A^{i}_{k}$). Then $A_{k} = \bigcup_{i=1}^{\sigma_{k} } A^{i}_{k}$,
 where ~$\sigma_{k} =  2^{(-k+1)(d-1)} n^{d-1}$. Fix any  ~$A^{i}_{k}$. Then ~$A^{i}_{k}$ is a translation of ~$(2^{k}, 2^{k+1})^{d}$. Applying Lemma ~\ref{sobolev} with ~$\Omega = (1,2)^d$, and ~$ \lambda = 2^k$ and using translation invariance, we have
\begin{equation*}
    \fint_{A^{i}_{k}} |u(x)-(u)_{A^{i}_{k}}|^{\tau}  dx \leq C^{\tau} \tau^{\frac{(d-s) \tau}{d}} [u]^{\tau}_{W^{s,p}(A^{i}_{k})}  ,
\end{equation*}
where ~$C= C(d,p)$ is a positive constant. Let ~$x = (x',x_d) \in A^{i}_{k}$.  Then ~$x_d \geq 2^k$ which implies ~$ \frac{1}{x_{d}} \leq \frac{1}{2^{k}}$. Therefore, using this and previous inequality, we have
\begin{equation*}
\begin{split}
    \int_{A^{i}_{k}} \frac{|u(x)|^{\tau}}{x^{d}_{d}}  dx 
   & \leq \frac{2^{\tau}}{2^{k d}} \int_{A^{i}_{k}} |u(x)-(u)_{A^{i}_{k}}|^{\tau}   dx + \frac{2^{\tau}}{2^{k d}} \int_{A^{i}_{k}} |(u)_{A^{i}_{k}}|^{\tau}dx \\ & \leq C^{\tau} \tau^{\frac{(d-s) \tau}{d}} [u]^{\tau}_{W^{s,p}(A^{i}_{k})} + 2^{\tau} |(u)_{A^{i}_{k}}|^{\tau} . 
\end{split}
    \end{equation*}
Summing the above inequality from ~$i=1$ to ~$\sigma_{k}$ and using the inequality ~\eqref{sumineq} with $\gamma= \frac{\tau}{p}$, we obtain
\begin{equation*}
\begin{split}
    \int_{A_{k}} \frac{|u(x)|^{\tau}}{x^{d}_{d}}  dx & \leq C^{\tau} \tau^{\frac{(d-s) \tau}{d}}   \sum_{i=1}^{\sigma_{k}} [u]^{\tau}_{W^{s,p}(A^{i}_{k})} + 2^{\tau}  \sum_{i=1}^{\sigma_{k}} |(u)_{A^{i}_{k}}|^{\tau}  \\ &
    \leq C^{\tau} \tau^{\frac{(d-s) \tau}{d}} [u]^{\tau}_{W^{s,p}(A_{k})} +  2^{\tau}  \sum_{i=1}^{\sigma_{k}} |(u)_{A^{i}_{k}}|^{\tau}.
\end{split}
\end{equation*}
Again, summing the above inequality from ~$k=m \in \mathbb{Z}^{-}$ to ~$-1$, we get
\begin{equation}\label{eqnn1}
\sum_{k=m}^{-1} \int_{A_{k}} \frac{|u(x)|^{\tau}}{x^{d}_{d}}  dx \leq  C^{\tau} \tau^{\frac{(d-s) \tau}{d}} \sum_{k=m}^{-1} [u]^{\tau}_{W^{s,p}(A_{k})} + 2^{\tau} \sum_{k=m}^{-1}  \sum_{i=1}^{\sigma_{k}} |(u)_{A^{i}_{k}}|^{\tau}  ,
\end{equation} 
where ~$C$ is a positive constant depends on ~$d$ and ~$p$. Independently, using triangle inequality, we have 
\begin{equation*}
    |(u)_{A^{j}_{k+1}}|^\tau \leq  \left( |(u)_{A^{i}_{k}}| + |(u)_{A^{i}_{k}} - (u)_{A^{j}_{k+1}}| \right)^\tau.
\end{equation*}
For ~$k \in \mathbb{Z}^{-} \backslash \{-1 \}$, applying \eqref{estimate} with ~$c :=c_{1}2^{d-1}>1$ where ~$c_{1} = \frac{2}{1+ 2^{d-1}} <1$ and using Lemma ~\ref{avg} with ~$E=A^{i}_{k}$ and $F=A^{j}_{k+1}$ with ~$sp=d$ together with Lemma ~\ref{sobolev}, we obtain
\begin{equation*}
    |(u)_{A^{j}_{k+1}}|^{\tau} \leq c_{1} 2^{d-1} |(u)_{A^{i}_{k}}|^{\tau} + C^{\tau} \tau^{\frac{(d-s) \tau}{d}} [u]^{\tau}_{W^{s,p}(A^{i}_{k} \cup A^{j}_{k+1})}  .
\end{equation*}
Since, there are ~$2^{d-1}$ such ~$A^{i}_{k}$'s cubes lies below the cube ~$A^{j}_{k+1}$. Therefore, summing the above inequality from ~$i=2^{d-1}(j-1)+1$ to ~$2^{d-1}j$, we obtain
\begin{equation*}
\begin{split}
     2^{d-1}  |(u)_{A^{j}_{k+1}}|^{\tau} & \leq c_{1} 2^{d-1}  \sum_{i=2^{d-1}(j-1)+1}^{2^{d-1}j}  |(u)_{A^{i}_{k}}|^{\tau} 
    \\  & \hspace{4mm}  + \ C^{\tau} \tau^{\frac{(d-s) \tau}{d}}  \sum_{i=2^{d-1}(j-1)+1}^{2^{d-1}j} [u]^{\tau}_{W^{s,p}(A^{i}_{k} \cup A^{j}_{k+1})}  .   
\end{split}
\end{equation*}
Again, summing the above inequality from ~$j=1$ to ~$\sigma_{k+1}$ and using the fact that
\begin{equation*}
  \sum_{j=1}^{\sigma_{k+1}} \Bigg(  \sum_{i=2^{d-1}(j-1)+1}^{2^{d-1}j}  |(u)_{A^{i}_{k}}|^{\tau} \Bigg)  = \sum_{i=1}^{\sigma_{k}}  |(u)_{A^{i}_{k}}|^{\tau},
\end{equation*}
we obtain
\begin{eqnarray*}
  \sum_{j=1}^{\sigma_{k+1}}  |(u)_{A^{j}_{k+1}}|^{\tau} &\leq& c_{1}  
\sum_{i=1}^{\sigma_{k}} |(u)_{A^{i}_{k}}|^{\tau} \\ && \hspace{3mm}  + \ C^{\tau} \tau^{\frac{(d-s) \tau}{d}} \sum_{j=1}^{\sigma_{k+1}} \Bigg( \sum_{i=2^{d-1}(j-1)+1}^{2^{d-1}j}  [u]^{\tau}_{W^{s,p}(A^{i}_{k} \cup A^{j}_{k+1})} \Bigg) \\ &\leq& c_{1}  
\sum_{i=1}^{\sigma_{k}} |(u)_{A^{i}_{k}}|^{\tau} 
  + \ C^{\tau} \tau^{\frac{(d-s) \tau}{d}} [u]^{\tau}_{W^{s,p}(A_{k} \cup A_{k+1})} .
\end{eqnarray*}
Summing the above inequality from ~$k=m \in \mathbb{Z}^{-}$ to ~$-2$, we get
\begin{equation*}
   \sum_{k=m}^{-2}  \sum_{j=1}^{\sigma_{k+1}}  |(u)_{A^{j}_{k+1}}|^{\tau}   \leq c_{1} \sum_{k=m}^{-2}  \sum_{i=1}^{\sigma_{k}} |(u)_{A^{i}_{k}}|^{\tau}  + C^{\tau} \tau^{\frac{(d-s) \tau}{d}}  \sum_{k=m}^{-2} [u]^{\tau}_{W^{s,p}(A_{k} \cup A_{k+1})}  .
\end{equation*}
By changing sides, rearranging, and re-indexing, we get
\begin{equation*}
 (1-c_{1}) \sum_{k=m+1}^{-1}  \sum_{i=1}^{\sigma_{k}} |(u)_{A^{i}_{k}}|^{\tau}  \leq   \sum_{j=1}^{\sigma_{m}} |(u)_{A^{j}_{m}}|^{\tau} +  C^{\tau} \tau^{\frac{(d-s) \tau}{d}} \sum_{k=m}^{-2} [u]^{\tau}_{W^{s,p}(A_{k} \cup A_{k+1})}  .
\end{equation*}
Choose ~$-m$ large enough such that ~$|(u)_{A^{j}_{m}}|=0$ for all ~$j \in \{ 1, \dots, \sigma_{m} \}$. Therefore, we have
\begin{equation}\label{eqnn}
    (1-c_{1})   \sum_{k=m}^{-1}  \sum_{i=1}^{\sigma_{k}} |(u)_{A^{i}_{k}}|^{\tau} \leq C^{\tau} \tau^{\frac{(d-s) \tau}{d}} \sum_{k=m}^{-2} [u]^{\tau}_{W^{s,p}(A_{k} \cup A_{k+1})}  .
\end{equation}
Combining ~\eqref{eqnn1} and ~\eqref{eqnn} and applying the inequality ~\eqref{sumineq} with ~$\gamma = \frac{\tau}{p}$ yields
\begin{equation*}
\begin{split}
    \sum_{k=m}^{-1} \int_{A_{k}} & \frac{|u(x)|^{\tau}}{x^{d}_{d} }   dx \leq  C^{\tau} \tau^{\frac{(d-s) \tau}{d}} \sum_{k=m}^{-2} [u]^{\tau}_{W^{s,p}(A_{k} \cup A_{k+1})} \\ &=  C^{\tau} \tau^{\frac{(d-s) \tau}{d}} \Bigg( \sum_{\substack{k =m \\ -k \ \text{is even}}}^{-2} [u]^{\tau}_{W^{s,p}(A_{k} \cup A_{k+1})}   + \sum_{\substack{k =m \\ -k \ \text{is odd}}}^{-2} [u]^{\tau}_{W^{s,p}(A_{k} \cup A_{k+1})} \Bigg) \\ &\leq 2 C^{\tau} \tau^{\frac{(d-s) \tau}{d}} [u]^{\tau}_{W^{s,p}(\Omega_{n})}  .
\end{split}
\end{equation*}
Therefore, we have
\begin{equation*}
    \sum_{k=m}^{-1} \int_{A_{k}} \frac{|u(x)|^{\tau}}{x^{d}_{d} }  dx  \leq C^{\tau} \tau^{\frac{(d-s) \tau}{d}} [u]^{\tau}_{W^{s,p}(\Omega_{n})}.
\end{equation*}
This proves the lemma.
\end{proof}

\smallskip

Using patching techniques along with Lemma ~\ref{flat case sp=1} and Lemma ~\ref{sp=d flat case}, Proposition ~\ref{prop} can be established (refer to ~\cite{AdiPurbPro2023} and ~\cite[Section 4]{adimurthi2023fractional}). For simplicity, we considered ~$\Omega = (0,1)$ for the case ~$d=1$ and proved Proposition ~\ref{prop}. For general interval $\Omega$ when ~$d=1$, the following proposition follows from translation and dilation of the interval ~$(0,1)$. Moreover, the positive constant $C$ in Proposition ~\ref{prop} does not depend on $\tau$. Therefore, we present the following proposition:

\begin{proposition}\label{prop}
 \normalfont   Let ~$\Omega$ be a bounded Lipschitz domain in $\mathbb{R}^{d}$ and ~$\sup_{x \in \Omega} \delta_{\Omega}(x) =R>0$. Suppose ~$p>1, ~ \tau \geq p, ~ s \in (0,1)$ such that ~$sp=d$. Then there exists a positive constant $C$ does not depend on ~$\tau$ such that
    \begin{equation}
      \left(  \int_{\Omega} X_{d}(u)^{\tau} \frac{dx}{\delta^{d}_{\Omega}(x)} \right)^{\frac{1}{\tau}} \leq C\tau^{\frac{d-s}{d}} [u]_{W^{s,p}(\Omega)} , \hspace{.3cm} \ \forall \ u \in W^{s,p}_{0}(\Omega),
    \end{equation}
    where $X_{d}(u)$ is defined in \eqref{condition on Xd}.
     \end{proposition}


\section{Proof of Theorem ~\ref{main result}}\label{proof of main result}
In this section we will prove Theorem ~\ref{main result}. We will establish a Trudinger-type inequality in the context of fractional boundary Hardy inequalities for the case ~$sp=d$ using the fact that the constant $C>0$ available in Proposition ~\ref{prop} is independent of ~$\tau$.  First we will establish the existence of ~$\alpha_{0}$ in Subsection ~\ref{subsec1}, then we will prove that ~\eqref{main inequality} fails for any ~$\alpha> \alpha^{*}_{d}$ in Subsection ~\ref{subsec2}. The existence of ~$\alpha_0$ in Theorem ~\ref{main result} is comparable to a result found in G. Leoni's book ~\cite[Theorem 7.19]{leonibook}, where he considers the space $W^{s,p}(\mathbb{R}^d \setminus \{ 0 \})$ with $sp = d$ and establishes a fractional version of Trudinger-type inequality using $\Phi_{n_{0}}$.

\smallskip

\subsection{Existence of \texorpdfstring{$\alpha_{0}$}{alpha not} in Theorem ~\ref{main result}}\label{subsec1}
Let ~$X_{d}(u)$ be defined in ~\eqref{condition on Xd}. By Proposition ~\ref{prop}, for every ~$\tau \geq p$ and ~$u \in W^{s,p}_{0}(\Omega)$ such that ~$[u]_{W^{s,p}(\Omega)} \leq 1$, we have the following inequality:
\begin{equation*}
    \left(  \int_{\Omega} X_{d}(u)^{\tau} \frac{dx}{\delta^{d}_{\Omega}(x)} \right)^{\frac{1}{\tau}} \leq C\tau^{\frac{d-s}{d}},
\end{equation*}
where ~$C= C(d,p,s, \Omega)$ is a positive constant. Now, for ~$\alpha > 0$, using  ~$\frac{nd}{d-s} \geq p = \frac{d}{s}$ for all ~$n \geq n_{0}$ as ~$n_{0} \geq \frac{d}{s}-1 > n_{0}-1 $ and from the definition of ~$\Phi_{n_{0}}$ defined in ~\eqref{notation}, we obtain
\begin{equation*}
\begin{split}
\int_{\Omega}  \Phi_{n_{0}} \left( \alpha X_{d}(u)^{\frac{d}{d-s}} \right)  \frac{dx}{\delta^{d}_{\Omega}(x)} & = \sum_{n=n_{0}}^{\infty} \frac{\alpha^{n}}{n!} \int_{\Omega} X_{d}(u)^{\frac{nd}{d-s}} \frac{dx}{\delta^{d}_{\Omega}(x)} \\ &  \leq \sum_{n=n_{0}}^{\infty} \frac{1}{n!} \left( \alpha C^{\frac{d}{d-s}}\frac{d}{d-s}  \right)^{n} n^{n}.
\end{split}
\end{equation*}
Using Stirling's approximation ~$n! \sim \sqrt{2 \pi n} \left( \frac{n}{e} \right)^{n}$ as ~$n \to \infty$, applying this approximation to the sum leads to the conclusion that there exists a sufficiently small ~$\alpha_0>0$, such that
\begin{equation*}
   \sup \left\{ \int_{\Omega} \Phi_{n_{0}} \left( \alpha X_{d}(u)^{\frac{d}{d-s}} \right) \frac{dx}{\delta^{d}_{\Omega}(x)} \ \bigg| \ u \in W^{s,p}_{0}(\Omega), \ [u]_{W^{s,p}(\Omega)} \leq 1 \right\}  < \infty  , \ \forall \ \alpha \in [0, \alpha_{0}).
\end{equation*}
This establishes the existence of ~$\alpha_{0}$ in Theorem ~\ref{main result}.

\bigskip

\noindent  \textbf{Proof of Corollary ~\ref{corollary 1}:}  We will now prove Corollary ~\ref{corollary 1}, which follows from  ~\eqref{main inequality} in Theorem ~\ref{main result}. Let ~$u \in W^{s,p}_{0}(\Omega)$ and $\alpha>0$. Since, ~$C^{\infty}_{c}(\Omega)$ is dense in ~$W^{s,p}_{0}(\Omega)$, there exists ~$v \in C^{\infty}_{c}(\Omega)$ and ~$w \in W^{s,p}_{0}(\Omega)$ such that ~$u:= v+ w$ and
    \begin{equation*}
        [u-v]_{W^{s,p}(\Omega)} = [w]_{W^{s,p}(\Omega)} \leq \frac{1}{2} \left( \frac{\alpha_{0}}{2\alpha} \right)^{\frac{d-s}{d}},
    \end{equation*} 
    where ~$\alpha_{0}$ is an defined in ~\eqref{main inequality}. Hence, applying the inequality
    \begin{equation*}
  X_{d}(u)^{\frac{nd}{d-s}} =  X_{d}(v+w)^{\frac{nd}{d-s}} \leq 2^{\frac{nd}{d-s}-1} X_{d}(v)^{\frac{nd}{d-s}} +  2^{\frac{nd}{d-s}-1} X_{d}(w)^{\frac{nd}{d-s}},   
    \end{equation*}
     where ~$  n \geq n_{0}$, we deduce that
    \begin{equation*}
        \Phi_{n_{0}} \left( \alpha X_{d}(u)^{\frac{d}{d-s}}  \right)  \leq \frac{1}{2} \Phi_{n_{0}} \left( 2^{\frac{d}{d-s}} \alpha   X_{d}(v)^{\frac{d}{d-s}} \right) + \frac{1}{2} \Phi_{n_{0}} \left( 2^{\frac{d}{d-s}} \alpha   X_{d}(w)^{\frac{d}{d-s}} \right).
    \end{equation*}
    Using  ~$[w]_{W^{s,p}(\Omega)} \leq \frac{1}{2} \left( \frac{\alpha_{0}}{2\alpha} \right)^{\frac{d-s}{d}}$ and multiplying the above inequality by $\frac{1}{\delta^{d}_{\Omega}(x)}$, we obtain
    \begin{equation}\label{eqn 113}
    \begin{split}
        \Phi_{n_{0}} \left( \alpha X_{d}(u)^{\frac{d}{d-s}} \right) \frac{1}{\delta^{d}_{\Omega}(x)} & \leq  \frac{1}{2} \Phi_{n_{0}} \left( 2^{\frac{d}{d-s}} \alpha   X_{d}(v)^{\frac{d}{d-s}} \right) \frac{1}{\delta^{d}_{\Omega}(x)} \\ & \hspace{4mm}  + \frac{1}{2} \Phi_{n_{0}} \left( \frac{\alpha_{0}}{2} \left( \frac{X_{d}(w)}{[w]_{W^{s,p}(\Omega)}} \right)^{\frac{d}{d-s}}   \right) \frac{1}{\delta^{d}_{\Omega}(x)}.  
    \end{split}
    \end{equation}
    Assume ~$d \geq 2$. Since, ~$v \in C^{\infty}_{c}(\Omega)$ which imply ~$supp(v)$ is compact. Therefore, there exists $C>0$ and $E \subset \Omega$ compact set such that $|v(x)| \leq C \chi_{E}(x)$ for all $x \in \Omega$, where $\chi_{E}$ is a characteristic function on $E$. We arrive at
\begin{equation*}
\begin{split}
    \Phi_{n_{0}} \left( 2^{\frac{d}{d-s}} \alpha   X_{d}(v)^{\frac{d}{d-s}} \right) \frac{1}{\delta^{d}_{\Omega}(x)}  &= \sum_{n=n_{0}}^{\infty} \frac{\alpha^{n}}{n!} (2|v(x)|)^{\frac{nd}{d-s}} \frac{1}{\delta^{d}_{\Omega}(x)} \\ & \leq \left( \sum_{n=n_{0}}^{\infty} \frac{\alpha^{n}}{n!} (2C)^{\frac{nd}{d-s}} \right) \frac{\chi_{E}(x)}{\delta^{d}_{\Omega}(x)}  \in L^{1}(\Omega).
\end{split}
\end{equation*}
Therefore, combining the above two estimates and using ~\eqref{main inequality} in Theorem ~\ref{main result} for the function ~$\frac{w}{[w]_{W^{s,p}(\Omega)}} \in W^{s,p}_{0}(\Omega)$, we conclude that for ~$d \geq 2$,
\begin{equation*}
    \Phi_{n_{0}} \left( \alpha X_{d}(u)^{\frac{d}{d-s}} \right) \frac{1}{\delta^{d}_{\Omega}(x)} \in L^{1}(\Omega). 
\end{equation*}
Now assume the case ~$d=1$. Since, $supp(v)$ is compact, there exists a constant $C>0$ such that $|v(x)-(v)_{\Omega}|^{\frac{1}{1-s}} \leq C$ for all $x \in \Omega$. Using this, we obtain
\begin{equation*}
\begin{split}
    \Phi_{n_{0}} \left( 2^{\frac{1}{1-s}} \alpha   X_{1}(v)^{\frac{1}{1-s}} \right) \frac{1}{\delta_{\Omega}(x)} &= \sum_{n=n_{0}}^{\infty} \frac{\left(2^{\frac{1}{1-s}} \alpha \right)^{n}}{n!}  \left( \frac{|v(x)-(v)_{\Omega}|}{\ln \left( \frac{2R}{\delta_{\Omega}(x)} \right)} \right)^{\frac{n}{1-s}} \frac{1}{\delta_{\Omega}(x)} \\ &\leq \sum_{n=n_{0}}^{\infty} \frac{\left(2^{\frac{1}{1-s}} \alpha C \right)^{n}}{n!}  \left( \frac{1}{\ln^{\frac{n}{1-s}} \left( \frac{2R}{\delta_{\Omega}(x)} \right)  } \right) \frac{1}{\delta_{\Omega}(x)} \\ &\leq \left( \sum_{n=n_{0}}^{\infty} \frac{\left( 2^{\frac{1}{1-s}} \alpha C \right)^{n}}{n!}  \right) \frac{1}{\ln^{\frac{1}{1-s}} \left( \frac{2R}{\delta_{\Omega}(x)} \right) \delta_{\Omega}(x) } .
\end{split}
\end{equation*}
Using the fact that $\frac{1}{\ln^{\frac{1}{1-s}} \left( \frac{2R}{\delta_{\Omega}(x)} \right) \delta_{\Omega}(x) } \in L^{1}(\Omega)$, we conclude 
\begin{equation*}
    \Phi_{n_{0}} \left( 2^{\frac{1}{1-s}} \alpha   X_{1}(v)^{\frac{1}{1-s}} \right) \frac{1}{\delta_{\Omega}(x)} \in L^{1}(\Omega).
\end{equation*}
Therefore, using this and ~\eqref{main inequality} in Theorem ~\ref{main result} for the function $\frac{w}{[w]_{W^{s,p}(\Omega)}} \in W^{s,p}_{0}(\Omega)$ to \eqref{eqn 113}, we conclude that
\begin{equation*}
     \Phi_{n_{0}} \left( \alpha X_{1}(u)^{\frac{1}{1-s}} \right) \frac{1}{\delta_{\Omega}(x)} \in L^{1}(\Omega). 
\end{equation*}
This finishes the proof of corollary.


\subsection{Existence of \texorpdfstring{$\alpha^{*}_{d}$}{alpha star} in Theorem ~\ref{main result}}\label{subsec2}  Define the family of Moser-type functions $u_{\epsilon}: \mathbb{R}^{d} \to \mathbb{R}$ by
\begin{equation}\label{moser-type fn}
   u_{\epsilon} (x) = \begin{dcases}
        |\ln \epsilon|^{\frac{d-s}{d}} , & |x| \leq \epsilon \\ 
       \frac{|\ln|x||}{|\ln \epsilon|^{\frac{s}{d}}}, & \epsilon < |x|<1  \\
       0, & |x| \geq 1.
    \end{dcases}
\end{equation}
\fbox{Case $d=1$} S. Iula in \cite[Proposition $2.2$]{lula2017} considered the above Moser-type functions \eqref{moser-type fn} for the case $sp=d=1$ and established that
 \begin{equation}\label{conv u epsilon d=1}
     [u_{\epsilon}]_{W^{s,p}(\mathbb{R})} \to   \left( 8 \Gamma(p+1) \sum_{n=0}^{\infty} \frac{1}{(1+2n)^{p}} \right)^{s} =: \gamma^{s}_{s}, \hspace{3mm} \text{as} \hspace{3mm} \epsilon \to 0.
 \end{equation}
For simplicity, we considered the interval ~$I := (-1,1)$ with ~$ \underset{x \in I}{\sup} \  \delta_{(-1,1)}(x)=R=1$ and establish that ~\eqref{main inequality} fails for any ~$\alpha> \alpha^{*}_{1}$. For general intervals, the result follows from dilation and translation of the interval. The Moser-type functions ~$u_{\epsilon}$ restricted to ~$I$, denoted as  ~$ u_{\epsilon} \big|_{I}$ and belongs to ~$W^{s,p}(I) = W^{s,p}_{0}(I)$  (see ~\cite[Theorem 6.78]{leonibook}).  Moreover, ~$(u_{\epsilon})_{I} \to 0$ as $\epsilon \to 0$. In particular, we have 
\begin{equation*}
\begin{split}
      (u_{\epsilon})_{I} =  \frac{1}{2} \int_{I} u_{\epsilon}(x)dx & = \int_{0}^{1} u_{\epsilon}(x) dx  = \epsilon |\ln \epsilon|^{1-s} + \frac{1}{|\ln \epsilon|^{s}} \int_{\epsilon}^{1} |\ln x| dx \\ & = \epsilon |\ln \epsilon|^{1-s} + \frac{1}{|\ln \epsilon |^{s}} \left(  1- \epsilon |\ln \epsilon| - \epsilon \right)  = \frac{1}{|\ln \epsilon|^{s}} (1- \epsilon).   
\end{split}
     \end{equation*}
     Therefore, 
     \begin{equation*}
         (u_{\epsilon})_{I} = \frac{1}{|\ln \epsilon|^{s}} (1- \epsilon) \to 0 \hspace{3mm} \text{as} \hspace{3mm} \epsilon \to 0.
     \end{equation*}
 We observe that $|\ln \epsilon| \to \infty$ as $\epsilon \to 0$. Therefore, using this and the fact that  $(u_{\epsilon})_{I} \to 0$ as $\epsilon \to 0$, for sufficiently small $\epsilon>0$, we have
\begin{equation*}
    | |\ln \epsilon|^{1-s} - (u_{\epsilon})_{I} | \geq \ln \left( \frac{2}{1- \epsilon} \right) |\ln \epsilon|^{1-s}.
\end{equation*} 
Also, $\ln \left( \frac{2}{1-x}  \right) < \ln \left( \frac{2}{1- \epsilon} \right)$ for all $x \in (0, \epsilon)$. Therefore, using this and the above inequality, we obtain
\begin{equation}\label{eqn 110}
 \frac{|u_{\epsilon}(x) - (u_{\epsilon})_{I} |}{\ln \left( \frac{2}{1-x}  \right)} =   \frac{||\ln \epsilon|^{1-s} - (u_{\epsilon})_{I} |}{\ln \left( \frac{2}{1-x}  \right)}   \geq  \frac{||\ln \epsilon|^{1-s} - (u_{\epsilon})_{I} |}{\ln \left( \frac{2}{1-\epsilon}  \right)} \geq  |\ln \epsilon|^{1-s},
\end{equation}
for all $x \in (0, \epsilon)$. 

\smallskip

Set $v_{\epsilon} := \frac{u_{\epsilon}}{[u_{\epsilon}]_{W^{s,p}(\mathbb{R})}}$. Observe that ~$[v_{\epsilon}]_{W^{s,p}(\mathbb{R})} = 1$ for all ~$\epsilon > 0$. Therefore, 
\begin{equation*}
    [v_{\epsilon}]_{W^{s,p}(I)} < [v_{\epsilon}]_{W^{s,p}(\mathbb{R})} = 1,
\end{equation*}
and ~$v_{\epsilon} \big|_{I} \in W^{s,p}(I) = W^{s,p}_{0}(I)$ (see ~\cite[Theorem 6.78]{leonibook}). From Corollary ~\ref{corollary 1} and Lemma ~\ref{lemma on tau < p} with $\frac{n}{1-s}<p$, where $1 \leq n< n_{0}$ and $n_{0} \geq p-1>n_{0}-1$, we deduce that for any $\alpha>0$,
\begin{equation*}
\begin{split}
    \bigintsss_{I} \Phi_{n_{0}} \Bigg( &  \alpha \frac{|v_{\epsilon}(x) - (v_{\epsilon})_{I}|^{\frac{1}{1-s}}}{\ln^{\frac{1}{1-s}} \left( \frac{2}{\delta_{I}(x)}  \right)} \Bigg) \frac{dx}{\delta_{I}(x)} \geq  \bigintsss_{I} \Phi_{n_{0}}  \Bigg(  \alpha \frac{|v_{\epsilon}(x) - (v_{\epsilon})_{I}|^{\frac{1}{1-s}}}{\ln^{\frac{1}{1-s}} \left( \frac{2}{\delta_{I}(x)}  \right)} \Bigg) dx \\ & =   \bigintsss_{I} \exp \Bigg(  \alpha \frac{|v_{\epsilon}(x) - (v_{\epsilon})_{I}|^{\frac{1}{1-s}}}{\ln^{\frac{1}{1-s}} \left( \frac{2}{\delta_{I}(x)}  \right)} \Bigg) dx - \sum_{n=0}^{n_{0}-1} \frac{\alpha^{n}}{n!} \bigintsss_{I}    \frac{|v_{\epsilon}(x) - (v_{\epsilon})_{I}|^{\frac{n}{1-s}}}{\ln^{\frac{n}{1-s}} \left( \frac{2}{\delta_{I}(x)} \right)  }  dx. 
\end{split}
\end{equation*}
Applying fractional Sobolev inequality  (see Lemma ~\ref{lemma on tau < p}) with  ~$[v_{\epsilon}]_{W^{s,p}(\Omega)}<1$ and $\frac{n}{1-s} <p$, where $1 \leq n < n_{0}$ and $n_{0} \geq p-1 >n_{0}-1$, we have
\begin{equation*}
    \bigintsss_{I}    \frac{|v_{\epsilon}(x) - (v_{\epsilon})_{I}|^{\frac{n}{1-s}}}{\ln^{\frac{n}{1-s}} \left( \frac{2}{\delta_{I}(x)} \right)  }  dx \leq C^{\frac{n}{1-s}}_{1} \int_{I} |v_{\epsilon}(x) - (v_{\epsilon})_{I}|^{\frac{n}{1-s}} dx \leq   \left(  \frac{C^{\frac{1}{1-s}}}{1-s} \right)^{n} n^{n}.
\end{equation*}
Here, we have used $\frac{1}{C_{1}} \leq \ln \left( \frac{2}{\delta_{I}(x)} \right)$ for all $x \in I$, for some $C_{1}= C_{1}(I)>0$. Using this in the previous inequality for $1 \leq n < n_{0}$ and $|I|=2$, we obtain
\begin{equation}\label{eqn1114}
\begin{split}
    \bigintsss_{I} \Phi_{n_{0}} \Bigg( &  \alpha \frac{|v_{\epsilon}(x) - (v_{\epsilon})_{I}|^{\frac{1}{1-s}}}{\ln^{\frac{1}{1-s}} \left( \frac{2}{\delta_{I}(x)}  \right)} \Bigg) \frac{dx}{\delta_{I}(x)} \\ & \geq \bigintsss_{I} \exp \left(  \alpha \frac{|v_{\epsilon}(x) - (v_{\epsilon})_{I}|^{\frac{1}{1-s}}}{\ln^{\frac{1}{1-s}} \left( \frac{2}{\delta_{I}(x)}  \right)} \right) dx - 2- \sum_{n=1}^{n_{0}-1} \frac{1}{n!}  \left(  \frac{\alpha C^{\frac{1}{1-s}}}{1-s} \right)^{n} n^{n}  .  
\end{split}
\end{equation}
 Let ~$\alpha^{*}_{1} := \gamma^{\frac{s}{1-s}}_{s}$. Fix any ~$\alpha > \gamma^{\frac{s}{1-s}}_{s}$. Since $[u_{\epsilon}]_{W^{s,p}(\mathbb{R})} \to \gamma^{s}_{s}$ as $\epsilon \to 0$ (see \eqref{conv u epsilon d=1}), for sufficiently small ~$\epsilon$, there exists $\beta>0$ such that $\alpha [u_{\epsilon}]^{\frac{-1}{1-s}}_{W^{s,p}(\mathbb{R})} \geq \beta >1$. For sufficiently small $\epsilon>0$,  we have
\begin{equation*}
\begin{split}
    \bigintsss_{I} \exp \left(  \alpha \frac{|v_{\epsilon}(x) - (v_{\epsilon})_{I}|^{\frac{1}{1-s}}}{\ln^{\frac{1}{1-s}} \left( \frac{2}{\delta_{I}(x)}  \right)} \right) dx  &\geq \bigintsss_{0}^{\epsilon} \exp \left(  \alpha \frac{|v_{\epsilon}(x) - (v_{\epsilon})_{I}|^{\frac{1}{1-s}}}{\ln^{\frac{1}{1-s}} \left( \frac{2}{1-x}  \right)} \right) dx \\ &= \bigintsss_{0}^{\epsilon} \exp \left(  \alpha [u_{\epsilon}]^{\frac{-1}{1-s}}_{W^{s,p}(\mathbb{R})} \frac{|u_{\epsilon}(x) - (u_{\epsilon})_{I}|^{\frac{1}{1-s}}}{\ln^{\frac{1}{1-s}} \left( \frac{2}{1-x}  \right)} \right) dx  \\  &\geq  \bigintsss_{0}^{\epsilon}  \exp \left(  \beta \frac{||\ln \epsilon|^{1-s} - (u_{\epsilon})_{I}|^{\frac{1}{1-s}} }{\ln^{\frac{1}{1-s}} \left( \frac{2}{1-x}  \right)} \right)  dx.
\end{split}
\end{equation*}
Again, for sufficiently small $\epsilon>0$, applying inequality ~\eqref{eqn 110} in the right hand side of the above inequality and noting that $\frac{1}{1-s}>1$,  we get
\begin{equation}\label{eqn1115}
  \bigintsss_{I} \exp \left(  \alpha \frac{|v_{\epsilon}(x) - (v_{\epsilon})_{I}|^{\frac{1}{1-s}}}{\ln^{\frac{1}{1-s}} \left( \frac{2}{\delta_{I}(x)}  \right)} \right) dx \geq   \int_{0}^{\epsilon}  \exp \left(  \beta |\ln \epsilon| \right)  dx =\epsilon^{1- \beta} \to \infty,
\end{equation}
as $\epsilon \to 0$, since $\beta>1$. Combining the inequalities \eqref{eqn1114} and \eqref{eqn1115}, we conclude that  
\begin{equation*}
     \sup \left\{  \int_{I} \Phi_{n_{0}} \left( \alpha X_{d}(u)^{\frac{1}{1-s}} \right) \frac{dx}{\delta_{I}(x)} \ \bigg| \ u \in W^{s,p}_{0}(I), \ [u]_{W^{s,p}(I)} \leq 1  \right\} =  \infty, \ \forall \ \alpha \in (\alpha^{*}_{1}, \infty). 
\end{equation*}

\bigskip

\noindent \fbox{Case $d \geq 2$} Let us consider the case ~$d \geq 2$ and ~$sp=d$.  E. Parini and B. Ruf in \cite[Proposition $5.2$]{ruf2019} using the above Moser-type functions ~\eqref{moser-type fn} established that there exists $\alpha^{*}_{d}$ as defined in ~\eqref{defn alpha star} such that for any $\alpha >  \alpha^{*}_{d}$, we have
\begin{equation}\label{eqn345}
    \sup \left\{  \int_{\Omega} \exp(\alpha |u(x)|^{\frac{d}{d-s}} ) dx \ \bigg| \ u \in \widetilde{W}^{s,p}_{0}(\Omega), \hspace{3mm} [u]_{W^{s,p}(\mathbb{R}^{d})} \leq 1  \right\} = \infty.
\end{equation}
Here, $\widetilde{W}^{s,p}_{0}(\Omega)$ denotes the completion of $C^{\infty}_{c}(\Omega)$ with respect to the norm $\|.\|_{W^{s,p}(\mathbb{R}^{d})}$. For bounded Lipschitz domains, L. Brasco, E. Lindgren, and E. Parini proved in \cite[Proposition B.1]{brasco2014} that $W^{s,p}_{0}(\Omega) = \widetilde{W}^{s,p}_{0}(\Omega)$ for $sp \neq 1$ with equivalent norms. Therefore, for any $u \in \widetilde{W}^{s,p}_{0}(\Omega)$ such that $[u]_{W^{s,p}(\mathbb{R}^{d})} \leq 1$ and $sp=d, ~ d \geq 2$, we have $u \in W^{s,p}_{0}(\Omega)$ and
\begin{equation*}
    [u]_{W^{s,p}(\Omega)} < [u]_{W^{s,p}(\mathbb{R}^{d})} \leq 1.
\end{equation*}
From Corollary ~\ref{corollary 1} and Lemma ~\ref{lemma on tau < p} with $\frac{nd}{d-s}<p$, where $1 \leq n< n_{0}$ and $n_{0} \geq p-1>n_{0}-1$, we deduce that for any $\alpha>0$,
\begin{equation*}
\begin{split}
    \int_{\Omega} \Phi_{n_{0}} \left( \alpha |u(x)|^{\frac{d}{d-s}}  \right) & \frac{dx}{\delta^{d}_{\Omega}(x)} \geq C_{1} \int_{\Omega} \Phi_{n_{0}} \left( \alpha |u(x)|^{\frac{d}{d-s}}  \right) dx \\ &  = C_{1} \int_{\Omega} \exp(\alpha |u(x)|^{\frac{d}{d-s}} ) dx - C_{1} \sum_{n=0}^{n_{0}-1} \frac{\alpha^{n}}{n!} \int_{\Omega} |u(x)|^{\frac{nd}{d-s}} dx  ,     
\end{split}
\end{equation*}
where ~$\frac{1}{\delta^{d}_{\Omega}(x)} \geq C_{1}$ for all $x \in \Omega$, for some ~$C_{1}>0$. Applying fractional Sobolev inequality for $sp=d, ~ d \geq 2$ (see Lemma \ref{lemma on tau < p})  with $[u]_{W^{s,p}(\Omega)} \leq 1$ and ~$\frac{nd}{d-s} <p$, where ~$1 \leq n < n_{0}$ and $n_{0} \geq p-1>n_{0}-1$, i.e., 
 \begin{equation}\label{eqn009}
     \int_{\Omega} |u(x)|^{\frac{nd}{d-s}} dx \leq \left(  C^{\frac{d}{d-s}} \frac{d}{d-s} \right)^{n} n^{n},
 \end{equation} 
 we get
 \begin{equation*}
 \begin{split}
   \int_{\Omega} \Phi_{n_{0}} \left( \alpha |u(x)|^{\frac{d}{d-s}}  \right) \frac{dx}{\delta^{d}_{\Omega}(x)} & \geq  C_{1} \int_{\Omega} \exp(\alpha |u(x)|^{\frac{d}{d-s}} ) dx - C_{1}|\Omega| \\ & \hspace{5mm}  - C_{1} \sum_{n=1}^{n_{0}-1} \frac{1}{n!}  \left( \alpha C^{\frac{d}{d-s}} \frac{d}{d-s} \right)^{n} n^{n}   .
 \end{split}
 \end{equation*}
Therefore, using the estimate \eqref{eqn345} with the above inequality and the fact $u \in W^{s,p}_{0}(\Omega)$ with $[u]_{W^{s,p}(\Omega)} < [u]_{W^{s,p}(\mathbb{R}^{d})} \leq 1$, we obtain that for any $\alpha> \alpha^{*}_{d}$,
\begin{equation*}
     \sup \left\{  \int_{\Omega} \Phi_{n_{0}}  \left( \alpha |u(x)|^{\frac{d}{d-s}} \right) \frac{dx}{\delta^{d}_{\Omega}(x)} \ \bigg| \ u \in W^{s,p}_{0}(\Omega), \hspace{3mm} [u]_{W^{s,p}(\Omega)} \leq 1  \right\} = \infty. 
\end{equation*}
This proves Theorem ~\ref{main result} for $d \geq 2$. \

\smallskip

\textbf{Proof of Remark:} \ Let $C=C(d,p,s, \Omega)>0$ be a uniform constant such that Lemma \ref{lemma on tau < p} for the case $d \geq 2$, Proposition ~\ref{prop} and the following inequality holds for the case $d=1$ with $sp=1$ and $\tau \geq 1$ (using Lemma \ref{lemma on tau < p}):
\begin{equation*}
\begin{split}
   \left( \bigintsss_{\Omega} \frac{|u(x) - (u)_{\Omega}|^{\tau}}{\ln^{\tau} \left( \frac{2R}{\delta_{\Omega}(x)}  \right)} dx \right)^{\frac{1}{\tau}} \leq C_{1} \left( \int_{\Omega} |u(x)-(u)_{\Omega}|^{\tau} dx  \right)^{\frac{1}{\tau}}  \leq C \tau^{\frac{d-s}{d}} & [u]_{W^{s,p}(\Omega)}, \\ & \forall \ u \in W^{s,p}_{0}(\Omega).
\end{split}
\end{equation*}
In the above inequality, we have used $\frac{1}{C_{1}} \leq \ln \left( \frac{2R}{\delta_{\Omega}(x)} \right)$ for all $x \in I$, for some $C_{1}= C_{1}(\Omega)>0$. Therefore, applying Lemma ~\ref{lemma on tau < p} for the case $d \geq 2$, the above inequality when $d=1$ and Proposition ~\ref{prop} with this constant $C= C(d,p,s, \Omega)>0$  along with $n_{0} \geq p-1>n_{0}-1$, we obtain for any $u \in W^{s,p}_{0}(\Omega)$  satisfying $[u]_{W^{s,p}(\Omega)} \leq 1$,
\begin{equation*}
\begin{split}
     \bigintsss_{\Omega} \Bigg( \sum_{n=0}^{n_{0}-1}  \frac{\alpha^{n} X_{d}(u)^{\frac{nd}{d-s}}}{n!} + &   \frac{\Phi_{n_{0}} \left( \alpha X_{d}(u)^{\frac{d}{d-s}} \right)}{\delta^{d}_{\Omega}(x)}  \Bigg)  dx \\ &  = \sum_{n=0}^{n_{0}-1} \frac{\alpha^{n}}{n!} \int_{\Omega} X_{d}(u)^{\frac{nd}{d-s}} dx +  \sum_{n=n_{0}}^{\infty} \frac{\alpha^{n}}{n!} \int_{\Omega} X_{d}(u)^{\frac{nd}{d-s}} \frac{dx}{\delta^{d}_{\Omega}(x)} \\ & \leq |\Omega| + \sum_{n=1}^{\infty} \frac{1}{n!} \left( \alpha C^{\frac{d}{d-s}}\frac{d}{d-s}  \right)^{n} n^{n}.
\end{split}
\end{equation*}
Using Stirling's approximation ~$n! \sim \sqrt{2 \pi n} \left( \frac{n}{e} \right)^{n}$ as ~$n \to \infty$, we conclude that there exists ~$\alpha_0>0$, such that
\begin{equation*}
\begin{split}
       \sup \Bigg\{  \bigintsss_{\Omega} \Bigg( \sum_{n=0}^{n_{0}-1}  \frac{\alpha^{n} X_{d}(u)^{\frac{nd}{d-s}}}{n!} +   \frac{\Phi_{n_{0}} \left( \alpha X_{d}(u)^{\frac{d}{d-s}} \right)}{\delta^{d}_{\Omega}(x)} \Bigg) dx \ \bigg|   u \in W^{s,p}_{0}(\Omega), & \ [u]_{W^{s,p}(\Omega)} \leq 1  \Bigg\} \\ &  < \infty, \ \forall \ \alpha \in [0, \alpha_{0}).
\end{split}
\end{equation*}
Also, using the fact
\begin{equation*}
    \left( \sum_{n=0}^{n_{0}-1}  \frac{\alpha^{n} X_{d}(u)^{\frac{nd}{d-s}}}{n!} + \frac{\Phi_{n_{0}} \left( \alpha X_{d}(u)^{\frac{d}{d-s}} \right)}{\delta^{d}_{\Omega}(x)} \right) \geq \frac{\Phi_{n_{0}} \left( \alpha X_{d}(u)^{\frac{d}{d-s}} \right)}{\delta^{d}_{\Omega}(x)}, \hspace{3mm} \forall \ u \in W^{s,p}_{0}(\Omega),
\end{equation*}
and the estimate ~\eqref{main ineq on alpha star} in Theorem ~\ref{main result},  we conclude that this supremum fails for any $\alpha > \alpha^{*}_d$.

\bigskip

\noindent \textbf{Open Problem:} It is easy to verify that for the case $sp=d, ~ d \geq 1$, the following inequality holds:
    \begin{equation*}
    \begin{split}
      \beta_{\delta_{\Omega}} := \sup & \left\{ \alpha >0  \  \bigg| \sup_{ [u]_{W^{s,p}(\Omega)}\leq 1 } \int_{\Omega} \Phi_{n_{0}} (\alpha X_{d}(u)^{\frac{d}{d-s}} ) \frac{dx}{\delta^{d}_{\Omega}(x)} < \infty \right\} \\ & \leq \beta := \sup \left\{ \alpha >0  \  \bigg| \sup_{ [u]_{W^{s,p}(\Omega)}\leq 1 } \int_{\Omega}\Phi_{n_{0}} (\alpha X_{d}(u)^{\frac{d}{d-s}} ) dx < \infty \right\}.   
    \end{split}
\end{equation*}
The most interesting open questions are whether there exists some $\alpha>0$ such that $\beta_{\delta_{\Omega}} < \alpha< \beta$ and the question of existence of extremal functions.
\section*{Acknowledgments}
We express our gratitude to the Department of Mathematics and Statistics at the Indian Institute of Technology Kanpur, India for providing conductive research environment. For this work, Adimurthi acknowledges support from IIT Kanpur. V. Sahu is grateful for the support received through MHRD, Government of India (GATE fellowship). This article forms a part of the doctoral thesis of V. Sahu. We would like to thank the anonymous referee, whose comments had helped us in significantly improving the results  of this article.

\end{document}